\documentclass[12pt]{amsart}

\newcommand{\rr}{\mathbb R}
\newcommand{\ball}{{B_r(x_0)}}
\newcommand{\ballo}{{B_r(0)}}
\newcommand{\omo}{\Omega_0}
\newcommand{\omt}{\Omega_t}

\newtheorem{theorem}{Theorem}
\newtheorem{lemma}[theorem]{Lemma}
\begin{document}

\title{Stability in $L^1$ of circular vortex patches}

\author{Thomas C.\ Sideris}
\address{Department of Mathematics\\ University of California\\ Santa Barbara, CA 93106\\ USA}
\email{sideris@math.ucsb.edu}
\thanks{T.C.S.\ was supported by a grant from the National Science Foundation.}

\author{Luis Vega}
\address{Departamento de Matem\'aticas\\ Universidad del Pa\'\i s Vasco\\ Apartado 644, 48080 Bilbao\\Spain\\}
\email{luis.vega@ehu.es}
\thanks{L.V. was supported by a grant from the Ministerio de Educaci\'on y Ciencia, MTM2007-62186.}

\thanks{The authors thank the anonymous referee for helpful comments.}


\date{May 28, 2009}

\begin{abstract}
The motion of incompressible and ideal fluids is studied in the plane.
The stability in $L^1$ of circular vortex patches is established among the class of all
bounded vortex
patches of equal strength.
\end{abstract}

\maketitle

For planar incompressible and ideal fluid flow, the theory of Yudovich \cite{yudo} establishes global
well-posedness of the initial value problem with initial vorticity in $L^1(\rr^2)\cap L^\infty(\rr^2)$.
Because vorticity is transported in 2d, it remains constant along particle trajectories.  If $\Phi_t$ is the flow map,
then the vorticity is given by $\omega(t,\Phi_t(y))=\omega(0,y)$, for all $t>0$ and $y\in\rr^2$.
When the initial vorticity is a patch of unit strength, represented by the indicator (characteristic) function $I_{\omo}$
of a bounded open set $\omo\subset\rr^2$, the resulting vorticity is $I_{\omt}$, with $\omt=\Phi_t(\omo)$.
In the special case when $\omo$ is equal to a ball $B$, the patch is stationary, $\Phi_t(B)=B$, for all $t>0$.
Theorem \ref{main}, our main result,   gives the stability in $L^1(\rr^2)$ of any circular patch  within the class of
all bounded vortex patches of equal strength.  No restriction is placed on the $L^1$ distance of the perturbation
to the ball, and the flow region is not limited to a bounded domain, but rather is the entire space $\rr^2$.

Wan and Pulvirenti \cite{pw} were the first to study stability  of vortex patches in $L^1$.  They considered the case
where the flow was contained in a bounded region, although for the stability of circular patches they mention  that
this assumption can be removed.
Their key estimate, {\bf (J)}, shows that the total angular momenta of the patches can be used to control the
$L^1$  difference between an arbitrary patch and a circular patch of the same total mass.  
They allow  the strengths of the patches differ, in which case the two patches are assumed to be close in $L^1$.
Our generalization of their inequality, given in Lemma \ref{ellone}, estimates the $L^1$ distance of an
arbitrary patch  to a circular patch, when both patches have equal strength.  Stability in $L^1$, given in Theorem \ref{main},
 follows immediately.
Weaker stability results were
given by Saffman \cite{saff} and Dritschel \cite{dritschel}.  Dritschel controls  the measure of the symmetric
difference of two patches through a convenient integral, and this idea is incorporated into our argument in Lemma \ref{set}.

Stability in $L^1$ does not imply that the boundaries of the two patches remain close in any metric.
Indeed, numerical simulations give strong evidence of fingering and filamentation, see \cite{buttke,dz}.  Spreading of vorticity
may also occur.  The best upper bound
for the  growth rate of the patch diameter is ${\mathcal O}(t\log t)^{1/4}$ given in \cite{isg}, see also \cite{march}.
 Nevertheless, in spite of the fact that the patch geometry may be complicated,
smoothness of smooth patch boundaries persists for all time, see  \cite{chemin}.

For any bounded open set $A\subset\rr^2$, denote its
mass, momentum, and angular momentum by
\[
|A|=\int_A dx,\quad M(A)=\int_A x\;dx,\quad\mbox{and}\quad i(A)=\int_A |x|^2dx,
\]
respectively.    Our arguments depend heavily upon the fact that these three quantities
are conserved in time
when $A=\omt$ is a patch moving with the flow.

\begin{lemma}\label{set}
If $A\subset\rr^2$ is any bounded open set, then
\[
i(A)-\frac{|A|^2}{2\pi}-\frac{|M(A)|^2}{|A|}\ge 0.
\]
Equality holds if and only if the set $A$ is a ball.
\end{lemma}

\begin{proof}
For any ball $\ball=\{x\in\rr^2: |x-x_0|<r\}$, introduce the quantity
\begin{equation}\label{qdef}
Q=Q(A;\ball)=
\int\limits_{A\triangle \ball} \left||x-x_0|^2-r^2\right|\;dx,
\end{equation}
in which $A\triangle \ball=(A\setminus\ball)\cup(\ball\setminus A)$ denotes the symmetric difference.
Note that $Q\ge0$ and $Q=0$ if and only if $A=\ball$.

The quantity $Q$ can also be written as
\[
Q=\int_A(|x-x_0|^2-r^2)\;dx+\int_\ball(r^2-|x-x_0|^2)\;dx,
\]
since the portions of these two integrals over the set $A\cap\ball$ cancel.

Now, we can expand the first integral in $Q$ and compute the second to obtain
\[
Q=i(A)-2x_0\cdot M(A)+(|x_0|^2-r^2)|A| +\frac\pi2r^4.
\]
A rearrangement of terms gives
\begin{equation}\label{identity}
Q=i(A)-\frac{|A|^2}{2\pi}-\frac{|M(A)|^2}{|A|}+\frac1{2\pi}\left(\pi r^2-|A|\right)^2+|A|\left|x_0-\frac{M(A)}{|A|}\right|^2.
\end{equation}
This last expression is minimized by choosing $\ball$ with the same mass and center of mass as $A$:
\[
|\ball|=\pi r^2=|A|\quad\mbox{and}\quad x_0=\frac{M(A)}{|A|}.
\]
With this choice, the Lemma now follows.
\end{proof}

\begin{lemma}\label{ellone}
If $B=\ballo$, then for any bounded open set $A$,
\[
\|I_{A}-I_B\|^2_{L^1}\le 4\pi\;
Q(A;B)
\]
in which $Q(A;B)$ is defined by \eqref{qdef}.  
Equality holds if and only if 
\begin{equation}
\label{specialcase}
A=B_a(0)\cup[B_b(0)\setminus\ballo],
\end{equation}
with $a<r<b$ and $r^2-a^2=b^2-r^2$.
\end{lemma}

\begin{proof}
Using  the identity \eqref{identity} and then Lemma \ref{set},
we have  for any bounded open set $A'$, 
\begin{equation}\label{prelim}
(|A'|-|B|)^2=(|A'|-\pi r^2)^2\le 2\pi\;Q(A';B),
\end{equation}
with equality if and only if $A'$ is a ball centered at the origin.

Next, we note that
\begin{align*}
\|I_{A}-I_B\|^2_{L^1}
&=|A \Delta B|^2\\
&=\left(|A\setminus B|+|B\setminus A|\right)^2\\
&\le2|A\setminus B|^2+2|B\setminus A|^2\\
&=2(|A\cup B|-|B|)^2+2(|A\cap B|-|B|)^2,
\end{align*}
with equality if and only if $|A\setminus B|=|B\setminus A|$.

Application of \eqref{prelim} with $A'=A\cup B$ and $A'=A\cap B$ yields
\begin{multline*}
2(|A\cup B|-|B|)^2+2(|A\cap B|-|B|)^2\\
\le 4\pi\;[Q(A\cup B;B)+Q(A\cap B;B)]=4\pi \; Q(A;B),
\end{multline*}
with equality if and only if $A\cup B$ and $A\cap B$ are balls centered at the origin.

This establishes the desired inequality.  The argument also shows that equality holds
if and only if $A\cup B=B_b(0)$, $A\cap B=B_a(0)$, with $a<r<b$, and
\[
|B_b(0)\setminus B|=|A\setminus B|=|B\setminus A|=|B\setminus B_a(0)|,
\]
which gives \eqref{specialcase}.
\end{proof}

\begin{theorem}\label{main}
Let $B=\ballo$.  Then for any bounded open set $\omo\subset\rr^2$, we have that
\[
\|I_{\omt}-I_B\|^2_{L^1}\le 4\pi\;\sup_{\omo\triangle B}\left||x|^2-r^2\right|\;\|I_{\omo}-I_B\|_{L^1},
\]
for all $t>0$.
\end{theorem}

\begin{proof}
The identity \eqref{identity} shows that  the quantity $Q(\omt;B)$ depends only on conserved quantities, and it is therefore also
conserved.   In other words, we have
 $Q(\omt;B)=Q(\omo;B)$, for all $t>0$.
Thus, the result follows from Lemma \ref{ellone} and the fact that
\[
Q(\omo;B)\le \sup_{\omo\triangle B}\left||x|^2-r^2\right|\;|\omo\triangle B|=\sup_{\omo\triangle B}\left||x|^2-r^2\right|\;\|I_{\omo}-I_B\|_{L^1}.
\]
\end{proof}

\end{document}